\documentclass[11pt,a4paper]{article}
\usepackage{amsmath, amsfonts, amssymb, color}
\usepackage{theorem}

\usepackage{graphicx}

\usepackage[latin1]{inputenc}

\newtheorem{thm}{Theorem}[section]
\newtheorem{lem}[thm]{Lemma}
\newtheorem{prop}[thm]{Proposition}

\newtheorem{rem}[thm]{Remark}

\def\be#1\ee{\begin{equation}#1\end{equation}}
\newcommand{\bea}{\begin{eqnarray}}
\newcommand{\eea}{\end{eqnarray}}
\newcommand{\beaa}{\begin{eqnarray*}}
\newcommand{\eeaa}{\end{eqnarray*}}
\newcommand{\bei}{\begin{itemize}}
\newcommand{\eei}{\end{itemize}}
\newcommand{\bee}{\begin{enumerate}}
\newcommand{\eee}{\end{enumerate}}

\def\P{{\mathbb{P}}}

\def\R{\mathbb{R}}
\def\E{\mathbb{E}\,}

\def\Z{{\mathbb Z}}
\def\N{{\mathbb N}}

\newcommand{\eps}{\varepsilon}
\newenvironment{proof}[1][] {\noindent {\bf Proof#1:} }{\hspace*{\fill}$\square$\medskip\par}

\def\ed#1{ {\mathbf 1}_{ \{#1  \}}}             

\def\cov{\textrm{cov}}

\def\A{{\mathcal A}}

\def\C{{\mathbb C}}

\def\DD{{\mathcal D}}

\def\FF{{\mathcal F}}
\def\KK{{\mathcal K}}

\def\MM{{\mathcal M}}

\def\QQ{{\mathcal Q}}
\def\TT{{\mathcal T}}

\def\XX{{\mathcal X}}

\def\sgn{{\textrm{sgn}}}

\def\i{{\bf i}}

\begin{document}

\title{\bf $L_2$-Small Deviations for Weighted Stationary Processes}
\author{
   Mikhail Lifshits
   \footnote{St.Petersburg State University, Russia, St. Petersburg,
   Universitetskii pr. 28,
   email {\tt mikhail@lifshits.org}.}
   \footnote{MAI, Link\"oping University.}
    \and
    Alexander Nazarov
    \footnote{St.Petersburg Department of Steklov Institute of Mathematics,
    email {\tt al.il.nazarov@gmail.com}.}
    \footnote{St.Petersburg State University, Russia, St. Petersburg, Universitetskii pr. 28.}
     }
\date{\today}

\maketitle

\begin{abstract}
We find logarithmic asymptotics of $L_2$-small deviation probabilities for
weighted stationary Gaussian processes (both for real and complex-valued) having power-type
discrete or continuous spectrum. Our results are based on the spectral theory
of pseudo-differential operators developed by Birman and Solomyak.
\end{abstract}



\section{Introduction}
\label{s:intro}

Let $(Y(t))_{t\in T}$ be a random process defined
on some parametric measure space $(T,\mathfrak m)$. Many studies have
been devoted to the asymptotic behavior of its $L_2$-small deviation
probabilities
\[
   \P\Big( ||Y||_2^2= \int\limits_T |Y(t)|^2 \mathfrak m(dt)
   \le \eps^2\Big), \qquad \textrm{ as } \eps\to 0,
\]
see e.g. \cite{DLL,GH1,GH2,Naz09,Naz11,NazNik,NazPus,Pus},
to mention just a small sample. The importance of small deviation probabilities
in a broader context and a large number of their applications
are described in the surveys \cite{LiSha,Lif}; for an extensive
up-to-date bibliography see \cite{Bib}.

In this work, we explore $L_2$-small deviation probabilities for
weighted stationary Gaussian processes having power-type spectrum.
Our goal is to relate the asymptotics of small deviation probabilities
with that of the spectrum.
From the historical point of view our results are closely related to those
on fractional Brownian motion and its relatives, see e.g.
\cite{Bro,Gen_IRMA,Gen_mem,LifLin3}. In terms of such processes with stationary
increments our message is that the spectral asymptotics is relevant to the small
deviation behavior but the self-similarity is not.

In Section \ref{s:per} we consider periodic processes that correspond to discrete
spectrum, while Section \ref{s:nonper} handles continuous time processes with spectral density.
The final results of two sections are quite similar, although intermediate technical details
differ.

Our results are based on the spectral theory of pseudo-differential operators
developed by Birman and Solomyak \cite{BS1,BS2}. This approach was initiated in \cite{HLN},
where a similar problem was considered in the discrete time setting.
In passing, in Section \ref{s:hln} we prove a slightly stronger version of one result
from \cite{HLN}.

The spectral results that we use are not sensible to the symmetry of the spectral measure.
Therefore, it is very natural to apply them to the complex-valued processes. In this context
{\it proper} Gaussian processes are particularly convenient because their distributions are
determined by the spectra of the corresponding covariance operators.
Our main results are logarithmic asymptotics of $L_2$-small deviation probabilities for
weighted stationary Gaussian processes  having power-type spectrum
in Theorem \ref{t:sd_w_per_real} (real-valued periodic process), in Theorem \ref{t:sd_w_per_compl}
(complex-valued proper periodic process), in Theorem \ref{t:sd_w_real} (real-valued process with
continuous spectrum), and in Theorem \ref{t:sd_w_compl} (complex-valued proper process with
continuous spectrum).

For the reader's convenience, in Appendix we formulate some particular cases of deep results
of \cite{BKS}--\cite{BS2} used in our proofs.

We denote by ${\mathcal F}$ the Fourier transform
\[
{\mathcal F}u(\xi)= \int\limits_{\R} \exp(-{\bf i}\xi x)u(x)~\!dx.
\]
For any two sequences $a_k, b_k$ the standard notation
$a_k\sim b_k$ means that $\lim_{k\to\infty} \tfrac{a_k}{b_k}=1$.

\section{Periodic stationary processes}
\label{s:per}
\setcounter{equation}{0}

\subsection{Spectral representations}

We first recall the necessary information on the spectral representations of
stationary periodic processes.

Let $X=\{X(t),t\in\R\}$ be a complex-valued $2\pi$-periodic centered second
order mean-square continuous stationary process.
Then its covariance function admits a spectral representation
\[
   K_X(s-t) := \cov(X(s),X(t))= \sum_{k\in\Z} \mu_k \,  e^{\i k (s-t)},
   \qquad s,t\in\R,
\]
where $\mu:=(\mu_k)_{k\in\Z}$ is a finite non-negative measure on $\Z$ called
the spectral measure of $X$.

The spectral representation  of $X$ itself writes as
\be\label{specX}
   X(t) =\sum_{k\in\Z} \sqrt{\mu_k} \, \xi_k \,  e^{\i k t},
\ee
where $\xi_k$ are centered uncorrelated complex random variables with
$\E|\xi_k|^2=1$.

Just for completeness, recall a straightforward reformulation for real-valued
processes. Let denote $\xi_k:=\xi_k^{(re)}+\i \xi_k^{(im)}$.
The process $X$ is real-valued iff
\begin{itemize}
\item $\xi_0$ is real;
\item $\mu_{-k}=\mu_k$ for all $k>0$;
\item $\xi_{-k}=\overline{\xi_k}$ for all $k>0$;
\item $\E|\xi_k^{(re)}|^2=\E|\xi_k^{(im)}|^2=1/2$ for all $k\in \Z$;
\item the real random variables $\left(\xi_0, (\xi_k^{(re)},\xi_k^{(im)})_{k>0}\right)$
are uncorrelated.
\end{itemize}
In this case \eqref{specX} writes as
\begin{eqnarray} 
  X(t)
\label{specXr}
   &=& \sqrt{\mu_0}\, \xi_0 +  \sum_{k=1}^\infty \sqrt{\mu_k} \,
        \left(\sqrt{2}\xi_k^{(re)} [\sqrt{2}\cos(kt)]
        - \sqrt{2}\xi_k^{(im)}[\sqrt{2}\sin(kt)] \right),
\end{eqnarray}
where the random variables $\sqrt{2}\xi_k^{(re)}, \sqrt{2}\xi_k^{(im)}$ have unit variance.

\subsection{Covariance operators and their factorization}

Let $\nu(du):=\tfrac{du}{2\pi}$ be the normalized Lebesgue measure on $[0,2\pi]$.
In the following, we will consider $X$ as a random element of $L_2([0,2\pi],\nu)$.
From this point of view, equations \eqref{specX} and \eqref{specXr} represent
the orthogonal expansions of $X$ with respect to the orthonormal bases
$(e^{\i k t})_{k\in \Z}$ and
$\left\{1,(\sqrt{2}\cos(kt),\sqrt{2}\sin(kt))_{k\ge 1}\right\} $,
respectively.
The elements of these bases are eigenvectors of the corresponding covariance operator
$\KK_X$ in $L_2([0,2\pi],\nu)$ and the corresponding eigenvalues are $\mu_k$.

The orthogonal expansions generate  natural decompositions of $\KK_X$. Let
$e_k:=\exp(\i k\cdot), k\in\Z$. Then the operator square root of $\KK_X$ is defined by the
formula $\DD e_k:=\sqrt{\mu_k}\, e_k$,  $k\in\Z$. Operator $\DD$ is bounded, self-adjoint,
satisfies $\DD\,\DD =\DD\, \DD^*=\KK_X$, and can be interpreted as a convolution operator with the kernel
\[
  \DD(s):= \sum_{\ell \in \Z} \sqrt{\mu_\ell} \,  e_\ell(s).
\]
Indeed, for every $k\in\Z$ and $s\in[0,2\pi]$ we have
\begin{eqnarray*}
   \int\limits_0^{2\pi} \DD(s-t)\, e_k(t)\, \nu(dt)
   &=& \sum_{\ell \in \Z}   \sqrt{\mu_\ell}\, \int\limits_0^{2\pi}  e_\ell(s-t)\, e_k(t)\, \nu(dt)
\\
   &=& \sum_{\ell \in \Z}  \sqrt{\mu_\ell} \ e_\ell(s) \int\limits_0^{2\pi}  e_{k-\ell}(t) \, \nu(dt)
   =  \sqrt{\mu_k}\ e_k(s).
\end{eqnarray*}

In the following we are interested in the small ball behavior of the {\it weighted} $L_2$-norm
\be\label{2pi}
  \int\limits_{0}^{2\pi} q(t) |X(t)|^2 dt = 2\pi ||\sqrt{q}X||^2_{2,\nu}
\ee
with some weight $q\in L_1[0,2\pi]$.

We have a decomposition for covariance operator
\be\label{decomp-cov}
   \KK_{\sqrt{q} X} = \QQ\KK_X \QQ= \QQ\DD\DD \QQ =: \TT^*\TT,\qquad \TT=\DD\,\QQ,
\ee
where $\QQ$ stands for the self-adjoint multiplication operator related to the function
$\sqrt{q}\in L_2[0,2\pi]$.
We claim that $\TT$ is the Hilbert--Schmidt operator although $\QQ$ need not be even bounded.
Indeed, since
$\sqrt{q}\in L_2[0,2\pi]$, it admits a Fourier series expansion
\[
   \sqrt{q}(t) := \sum_{m \in \Z} q_m \, e^{\i m t},
\]
with $(q_m)\in\ell_2(\Z)$.
Then we have
\[
  \TT e_\ell= \DD\left[ \sum_{k \in \Z} q_{k-\ell} e_k  \right]=
\sum_{k \in \Z} \sqrt{\mu_k}\, q_{k-\ell}\, g_\ell e_k,
\]
and therefore
\[
\sum_{\ell\in \Z} ||\TT e_\ell||_{2,\nu}^2 =  \sum_{\ell\in \Z}\sum_{k \in \Z} \mu_k \, |q_{k-\ell}|^2
=  \mu(\Z)\ ||q||_{1,\nu}.
\]

For the study of logarithmic asymptotics of small deviation probabilities, we need to know the
main term of the eigenvalue asymptotics for $\KK_{\sqrt{q}X}$, see \cite{Naz09}.
Since $K_{\sqrt{q}X}$ is non-negative, its eigenvalues $\lambda_n(\KK_{\sqrt{q}X})$ coincide
 with its singular values $s_n(\KK_{\sqrt{q}X})$.
We always label the eigenvalues and singular values in non-increasing order
counting multiplicity.

We study the {\it distribution function of singular values}
\[
  N(\lambda;\KK_{\sqrt{q}X}) :=  \#\{n: s_n(\KK_{\sqrt{q}X})\ge \lambda\},\
  \quad \lambda >0,
\]
and its asymptotics as $\lambda\to 0$. For a compact operator $T$ we introduce the notation
\[
\Delta_\theta(T) = \lim_{\lambda\to 0_+} \lambda^\theta N(\lambda; T).
\]
The following relation is important in what follows:
\be \label{Nass}
    \Delta_\theta(T) =\Delta \quad
    \Longleftrightarrow \quad
    \lim\limits_{n\to\infty} n^{\frac 1\theta}s_n(T)= \Delta^{\frac 1\theta}.
\ee

\subsection{Spectral asymptotics}

From now on we assume that the spectral measure has a power-like decay
\be \label{muk}
   \lim\limits_{k\to \pm\infty} |k|^r\mu_k=M_{\pm},
\ee
with some $r>1$ and $M_{\pm}\ge0$, $M_+ + M_->0$.
Assumption \eqref{muk} is typical of the literature on small deviations of Gaussian
processes; see for example \cite{DLL}.

\begin{lem}\label{Lemma21}
Let the spectral measure of $X$ satisfy $(\ref{muk})$, and let $q\in L_1[0,2\pi]$.
Then
\be \label{s2}
   \lambda_n(\KK_{\sqrt{q}X})=s_n(\KK_{\sqrt{q}X})\sim
   \left(\frac{M_-^{\frac 1r}+ M_+^{\frac 1r}}{2\pi}
   \int\limits_{0}^{2\pi} q(t)^{\frac 1r} \, dt\right)^{r}
   \,  n^{-r}, \qquad \textrm{as } n\to\infty.
\ee
\end{lem}

\begin{proof}
We proceed similar to \cite{HLN} where we used quite general results
of Birman and Solomyak \cite{BS1, BS2}.

We can consider $\KK_{\sqrt{q}X}$ as an operator in $L_2(\R)$
\[
  (\KK_{\sqrt{q}X}u)(s)= b(s)\int\limits_{\R}\frac 1{2\pi}\,K_X(s-t)b(t)u(t)\,dt,
\]
where $b=\sqrt{q}\cdot{\mathbf 1}_{[0,2\pi]}$. Notice that since we are working on the interval of length $2\pi$,
it is sufficient to consider only the restriction of our periodic
function $K_X$ to $[-2\pi,2\pi]$.

Let $h$ be a smooth cut-off function such that $h(t) = 1$ if $t\in [\tfrac{3\pi}{2},2\pi]$
and $h(t) = 0$ if $t\in [-2\pi,\pi]$. Then it follows that the function
$h_0(s):=1- h(s)-h(-s)$ equals one on $[-\pi,\pi]$ and vanishes outside
of the interval  $[-\tfrac{3\pi}{2},\tfrac{3\pi}{2}]$.
We decompose the kernel $K_X$ as follows:
\be \label{decomp}
   \frac 1{2\pi}\,K_X(s)=\frac 1{2\pi}\,K_X(s)\big[h(s)+h(-s)+h_0(s)\big]
   =:K_+(s)+K_-(s)+K^{(0)}(s),
\ee
and claim that the function $K^{(0)}$ satisfies\be \label{op}
   \lim\limits_{\xi\to\pm\infty}|\xi|^r\cdot\FF{K^{(0)}}(\xi)= M(\sgn(\xi)),
\ee
where $M(\pm1)=M_{\pm}$.
Indeed, we have
\[
   \frac 1{2\pi}\FF(K_X\cdot h_0)(\xi)=\frac 1{2\pi}\sum_{k\in\Z} \mu_k\FF h_0(\xi-k),
\]
and then by splitting the series into two sums, we obtain
\[
  \FF{K^{(0)}}(\xi)=\Sigma_1+\Sigma_2
  :=\left(\sum_{|k-\xi|\le\sqrt{\xi}}+\sum_{|k-\xi|>\sqrt{\xi}}\right)
  \frac {\mu_k}{2\pi}\,\FF h_0(\xi-k).
\]
Since $\FF h_0$ rapidly decays at infinity, we have
$\Sigma_2=o(|\xi|^{-r})$ as $|\xi|\to\infty$. Furthermore, (\ref{muk}) implies
that
\begin{eqnarray*}
  \Sigma_1 &=& \frac {M(\sgn(\xi))}{2\pi}\, |\xi|^{-r}
  \sum_{|k-\xi|\le\sqrt{\xi}}\FF h_0(\xi-k)+o(|\xi|^{-r})
  \\
  &=& \frac {M(\sgn(\xi))}{2\pi}\, |\xi|^{-r}\sum_k\FF h_0(\xi-k)
     +o(|\xi|^{-r})=M(\sgn(\xi)) |\xi|^{-r}+o(|\xi|^{-r})
\end{eqnarray*}
by the Poisson summation formula (see, e.g.,
\cite[Ch. II, Sect. 13]{Zyg}), so that (\ref{op}) follows.

 Now we introduce a model operator
\[
  (\A u)(s)= b(s)\FF^{-1} \big(a(\xi)\FF(bu)(\xi)\big),
\]
with
 \[
  a(\xi)=\zeta(\xi)M(\sgn(\xi))|\xi|^{-r},
 \]
where $\zeta$ is a  smooth cut-off function vanishing in a neighborhood of the origin. Since $b\in L_2$,
Proposition \ref{BS1} can be applied to the operator $\A$. This gives
\[
   \Delta_{\frac{1}{r}}(\A)
 =  \frac 1{2\pi} \int\limits_0^{2\pi} \int\limits_{\R\backslash\{0\}}
       \ed{|b(t)|^2\, |M(\sgn(\xi))| \, |\xi|^{-r} \ge 1} \,d\xi dt
     =  \frac 1{2\pi} \int\limits_0^{2\pi}  q(t)^{\frac{1}{r}} \,dt
       \left( M(-1)^{\frac{1}{r}}+ M(1)^{\frac{1}{r}} \right).
\]
Furthermore, the decomposition (\ref{decomp}) generates the corresponding
operator decomposition
\[
   \KK_{\sqrt{q}X}=\KK_+ +\KK_- +\KK^{(0)}.
\]
Since the relation (\ref{op}) implies $\FF{K^{(0)}}(\xi)-a(\xi)=o(|\xi|^{-r})$ as $\xi\to\infty$,
part {\bf 1} of Proposition \ref{BS2} gives
$\Delta_{\frac{1}{r}} (\KK^{(0)}-\A)=0$.
Moreover, since
$K_X$ is $2\pi$-periodic, the singular values of $\KK_+$ coincide
with the singular values of the operator
\[
    b(s+\pi) {\mathbf 1}_{[0,\pi]}(s)
    \int\limits_{\R}
    \frac 1{2\pi}K_X(s-t)h(s+2\pi-t) b(t){\mathbf 1}_{[\pi,2\pi]}(t) u(t) \,dt.
 \]
For this operator, the support of the ``left'' weight is $[0,\pi]$, and the support
of the ``right'' weight is $[\pi, 2\pi]$. Part {\bf 2} in Proposition \ref{BS2}
gives $\Delta_{\frac{1}{r}} (\KK_+)=0$. Similarly,
$\Delta_{\frac{1}{r}}(\KK_-)=0$. By Proposition \ref{BS3} we obtain
\[
  \Delta_{\frac{1}{r}}(\mathcal K_{\sqrt{q}X})
= \Delta_{\frac{1}{r}} (\mathcal K^{(0)}) = \Delta_{\frac{1}{r}}(\mathcal A),
\]
and the equivalence in \eqref{Nass} gives (\ref{s2}).
\end{proof}

\subsection{Gaussian small deviations}

Now we transform the information about the eigenvalues into that on
small deviation asymptotic behavior. This can be done for real processes and
also for an important class of complex processes. We handle two cases separately
because the constants appearing in the results are slightly different.

\subsubsection{Real processes}

Recall that if we have a centered Gaussian random vector $Z$,\
in a real Hilbert space, and $\KK_Z$ stands for its covariance operator,
then, by the Karhunen--Lo\`eve expansion (see \cite[Section 1.4]{AsG}),
\[
    ||Z||^2= \sum_{n=1}^\infty \lambda_n \, \xi_n^2\ ,
\]
where $(\xi_n)_{n\in\N}$ is a sequence of independent
standard normal random variables and
$(\lambda_n)_{n\in\N}$ are the eigenvalues of $\KK_Z$.
Therefore, the sequence $(\lambda_n)_{n\in\N}$ determines the distribution
of $||Z||$. In particular, if
\be \label{lambda}
     \lambda_n  \sim  C\, n^{-r},
     \qquad \textrm{ as } n\to\infty,
\ee
then it is well known  from \cite[p.67]{DLL} or \cite{Zol}, that
\be \label{sd2}
   \ln \P\left(||Z|| \le \eps\right) \sim - B_r \,(C/\eps^2)^{\frac{1}{r-1}}
\ee
with $C$ from \eqref{lambda} and $B_r := \tfrac{r-1}{2}
     \left(\tfrac{\pi}{r\sin(\pi/r)}\right)^{\frac{r}{r-1}} $.
If our process $X$ is real, we can apply the formula \eqref{sd2} to $\sqrt{q}X$
considered as an element
of $L_2([0,2\pi],\nu)$ and using eigenvalue asymptotics \eqref{s2} as \eqref{lambda}.
Notice that for real processes the spectral measure is symmetric, i.e. we have
$M_+=M_-:=M$. Taking into account (\ref{2pi}) we immediately obtain the following result.

\begin{thm} \label{t:sd_w_per_real}
Let $\{X(t), t\in\R\}$ be a $2\pi$-periodic real centered
mean-square continuous stationary Gaussian process.
Assume that its spectral measure satisfies the asymptotic condition
\[
    \mu_k \sim M |k|^{-r}, \qquad \textrm{as } |k|\to \infty,
\]
with some $r>1, M>0$. Let $q$ be a summable weight.

Then we have, as $\eps\to 0$,
\[
   \ln \P\left( \int\limits_0^{2\pi} q(t) |X(t)|^2 dt \le \eps^2\right)
   \sim - \left(\frac{M^{\frac 1r}}{r\sin(\pi/r)}
   \int\limits_{0}^{2\pi} q(t)^{\frac 1{r}}  dt\right)^{\frac{r}{r-1}}\,
   \frac{(r-1)(2\pi)^{\frac{1}{r-1}}}{2\, \eps^{\frac{2}{r-1}}} \, .
 \]
\end{thm}

\subsubsection{Examples}
Consider the Bogoliubov process \cite{Pus, San} -- a $1$-periodic centered
stationary Gaussian process (with parameter $\omega>0$) defined by
\[
  \beta^{(\omega)}(s) :=  \sqrt{\mu_0}\, \xi_0 + \sum_{k=1}^\infty \sqrt{\mu_k}
     \left( \xi_k [\sqrt{2}\cos(2\pi ks)]+ \zeta_k [\sqrt{2}\sin(2\pi ks)]\right),
     \qquad s\in \R,
\]
with independent standard normal random variables $\left((\xi_k)_{k\ge 0},(\zeta_k)_{k>0} \right)$
and $\mu_k= \tfrac{1}{\omega^2+(2\pi k)^2}$.
Define a $2\pi$-periodic process $X(t):=\beta^{(\omega)}\left(t/2\pi)\right)$, $t\in \R$.
In our notation, for the spectrum of $X$ we have $r=2$, $M=\frac{1}{(2\pi)^2}$.
By applying Theorem \ref{t:sd_w_per_real} we obtain for $q\in L_1[0,1]$
\begin{eqnarray*}
   \ln \P\left( \int\limits_0^{1} q(s) |\beta^{(\omega)}(s)|^2 ds \le \eps^2\right)
   &=&  \ln \P\left( \int\limits_0^{2\pi} q(t/2\pi) |X(t)|^2 dt \le (\eps\, \sqrt{2\pi})^2\right)
\\
   &\sim& -  \frac 18 \left(\int\limits_{0}^{1} \sqrt{q(s)}\,ds\right)^2 \eps^{-2}.
\end{eqnarray*}
In the simplest case $q(s)\equiv 1$ we have
\[
  \ln \P\left( \int\limits_0^{1} |\beta^{(\omega)}(s)|^2 ds \le \eps^2\right)
   \sim -  \frac 18 \, \eps^{-2},
\]
cf. \cite[Theorem 1]{Pus}.

For $q(s)=e^{2as}, a\not=0$, our result gives
 \[
  \ln \P\left( \int\limits_0^{1} e^{2as} |\beta^{(\omega)}(s)|^2 ds \le \eps^2\right)
\sim- \frac 1{8} \ \left[\frac{e^a-1}{a}\right]^2\  \eps^{-2},
\]
as proved in \cite[Theorem 2]{Pus}.
\medskip

Our next example is the so-called $m$-times integrated-centered Brownian bridge.
Let $B_0(\tau)$ be standard Brownian bridge on $[0,1]$. We define the sequence
of Gaussian processes
\[
       B_{\{m\}}(s)=B_{m-1}(s)-\int\limits_0^1B_{m-1}(\tau)\,d\tau;\qquad B_m(\tau)
       =\int\limits_0^{\tau} B_{\{m\}}(s)\,ds,\qquad m\in\mathbb N.
\]
It was shown in \cite[Sec. 3]{Naz11} that
\[
  B_{\{m\}}(s)=  \sum_{k=1}^\infty (2\pi k)^{-m}
     \left( \xi_k [\sqrt{2}\cos(2\pi ks)]+ \zeta_k [\sqrt{2}\sin(2\pi ks)]\right),
     \qquad s\in [0,1],
\]
with independent standard normal random variables $\left((\xi_k)_{k>0},(\zeta_k)_{k>0} \right)$.
This formula obviously defines a $1$-periodic centered
stationary Gaussian process on $\R$. Define a $2\pi$-periodic process $X_m(t)=B_{\{m\}}(t/2\pi)$.
Then for the spectrum of $X_m$ we have $r=2m$, $M=(2\pi k)^{-2m}$.
By applying  Theorem \ref{t:sd_w_per_real} we obtain for $q\in L_1[0,1]$
\begin{eqnarray*}
   \ln \P\left( \int\limits_0^{1} q(s) |B_{\{m\}}(s)|^2 ds \le \eps^2\right)
   &=&  \ln \P\left( \int\limits_0^{2\pi} q(t/2\pi) |X_m(t)|^2 dt \le (\eps\, \sqrt{2\pi})^2\right)
\\
   &\sim& -  \frac {2m-1}2 \left(\frac 1{2m\sin(\pi/2m)}
   \int\limits_{0}^{1} q(s)^{\frac 1{2m}}\,ds\right)^{\frac {2m}{2m-1}} \eps^{-\frac 2{2m-1}}.
\end{eqnarray*}
For $q(s)\equiv 1$ this result agrees with \cite[Theorem 3.2]{Naz11}.

\begin{rem}
In fact, the {\it sharp} small ball asymptotics
for these processes
were obtained in $\cite{Naz11}$ and  $\cite{Pus}$,
see also $\cite{NazPus}$
 for more general weights. However, this is strongly connected with the
fact that $\beta^{(\omega)}$ and $B_{\{m\}}$ are {\rm the Green Gaussian processes} i.e. their covariances are
the Green functions for ordinary differential operators. In general case this seems to be a
much harder problem.
\end{rem}

\subsubsection{Proper complex processes}

If we have a centered Gaussian random vector $Z$
in a complex Hilbert space, and $\KK_Z$ stands for its covariance operator,
then Karhunen--Lo\`eve expansion yields
\be \label{KL-complex}
    Z= \sum_{n=1}^\infty \lambda_n \, \xi_n \, e_n\, ,
\ee
where $(\xi_n)_{n\in\N}$ are uncorrelated complex jointly Gaussian random
variables satisfying $\E|\xi_n|^2=1$ and
$(\lambda_n, e_n)_{n\in\N}$ are the eigenpairs of $\KK_Z$.
We still have
\be \label{norm_complex}
    ||Z||^2= \sum_{n=1}^\infty \lambda_n \, |\xi_n|^2\, ,
\ee
but, unfortunately, unlike the real case, the variables $\xi_n$ need not be independent,
although they are uncorrelated. Indeed, the independence of two centered complex Gaussian
random variables $\eta_1$ and $\eta_2$ is equivalent to the {\it pair}
of relations
\[
  \begin{cases}
   \cov(\eta_1,\eta_2)= \E \eta_1\overline{\eta_2}=0;&\cr
    \E \eta_1\eta_2=0.&
  \end{cases}
\]
Therefore, the sequence $(\lambda_n)_{n\in\N}$ {\it does not} determine the distribution
of $||Z||$ in general case. For this reason, we need to restrict the consideration to
an important subclass of the variables and processes where uncorrelated variables are
independent, cf. \cite{NM, Ol}.

A complex-valued random process $(X(t))_{t\in T}$ is called centered {\it proper} (or {\it circularly})
Gaussian if
\begin{itemize}
\item For any $t_1,\dots,t_n\in T$ the  coordinate vector
$\left(X^{(re)}(t_1), X^{(im)}(t_1),\dots, X^{(im)}{(t_n)}\right)$
is a centered Gaussian vector in $\R^{2n}$;
\item $\E X(t_1) X(t_2) = 0$ for all $t_1,t_2\in T$.
\end{itemize}
We clearly have $\E X(t)=0, \forall t\in T$. Moreover,
the property $\E X(t)^2=0$ yields that the distribution of $X(t)$ in
the complex plane $\C$ is spherically symmetric.

These properties extend to the span of $X$. Let us denote
$\XX:= \overline{\textup{span}\{X(t),t\in T\}}$.
For every $Y\in \XX$ we have $\E Y=0$, $\E Y^2=0$, hence its distribution in
$\C$ is spherically symmetric Gaussian.
Moreover, for any $Y_1,Y_2\in \XX$ we have $\E Y_1Y_2=0$ and $Y_1,Y_2$ are
independent iff they are uncorrelated, i.e. $\E Y_1\overline{Y_2}=0$. This
can be easily verified by checking that their coordinates are uncorrelated.

By applying these facts to the expansion  \eqref{KL-complex}
of a proper Gaussian process $Z$, we see that the variables $(\xi_n)_{n\in\N}$
are independent and spherically symmetric. Therefore, \eqref{norm_complex} becomes
\[
    ||Z||^2= \sum_{n=1}^\infty \frac{\lambda_n}{2} \,
    \left(\xi_{n,1}^2+ \xi_{n,2}^2 \right)\, ,
\]
where $(\xi_{n,j})_{n\in\N, j\in \{1,2\}}$  are independent real
standard Gaussian random variables. This formula can be rewritten as
\[
    ||Z||^2= \sum_{n=1}^\infty \lambda_n^* \xi_{n}^{*2},
\]
where
\begin{eqnarray*}
 \lambda_{2n-1}^* &:=& \frac{\lambda_n}{2}; \quad \xi_{2n-1}^*:= \xi_{n,1},
 \\
 \lambda_{2n}^* &:=& \frac{\lambda_n}{2}; \quad \xi_{2n}^*:= \xi_{n,2},
\end{eqnarray*}
for all $n\ge 1$.

A straightforward calculation shows that
$\lambda_n\sim C n^{-r}$ yields $\lambda_n^*\sim 2^{r-1}C n^{-r}$, as $n\to\infty$.
By applying \eqref{sd2} with $2^{r-1}C$ instead of $C$ we obtain the following result.

\begin{thm} \label{t:sd_w_per_compl}
Let $\{X(t), t\in\R\}$ be a $2\pi$-periodic complex centered
mean-square continuous stationary proper Gaussian process.
Assume that its spectral measure satisfies the asymptotic condition \eqref{muk}
with some $r>1$. Let $q$ be a summable weight.

Then we have, as $\eps\to 0$,

\[
   \ln \P\left( \int\limits_0^{2\pi} q(t) |X(t)|^2 dt \le \eps^2\right)
   \sim - \left(\frac{M_-^{\frac 1r}+ M_+^{\frac 1r}}{2r\sin(\pi/r)}
   \int\limits_{0}^{2\pi} q(t)^{\frac 1{r}}  dt\right)^{\frac{r}{r-1}}\,
   \frac{(r-1)(2\pi)^{\frac{1}{r-1}}}{\eps^{\frac{2}{r-1}}}.
\]
\end{thm}

\section{Stationary sequences}
\label{s:hln}
\setcounter{equation}{0}

Let a real stationary centered Gaussian sequence $(U_k)_{k\in\Z}$ admit a representation
\be \label{Uk}
   U_k = \sum_{m=-\infty}^\infty a_m X_{k-m},
\ee
where $(a_m)\in \ell_2(\Z)$, and $X_m$ are independent standard Gaussian random variables
(this representation exists iff  $(U_k)$ has a spectral density).

The following result was essentially obtained in \cite{HLN}.

\begin{thm} \label{t:sd_statseq}
Let a real stationary centered Gaussian sequence $(U_k)_{k\in\Z}$ admit a representation
\eqref{Uk}
and let the coefficients $(d_k)_{k\in\Z}$ have the asymptotics
\[
     \lim\limits_{k\to \pm\infty} |k|^p d_k=d_{\pm},\quad  \textrm{for some } p>\frac{1}{2} \, ,
\]
where at least one of the numbers $d_{\pm}$ is strictly positive.
Then, as $\eps\to 0$,
\[
   \ln \P\left( \sum_{k\in\Z} d_k^2 U_k^2 \le \eps^2\right)
   \sim - \left(
      \frac{d_-^{\frac 1p}+ d_+^{\frac 1p}} {4 \, p\, \sin\big(\frac{\pi}{2p}\big)}
      \
      \int\limits_0^{2\pi}
      |\mathfrak a(t)|^{\frac 1p} dt
      \right)^{\frac{2p}{2p-1}}
        \, \frac{2p-1}{2\, \eps^{\frac 2{2p-1}}}\, ,
\]
where ${\mathfrak a}(t)=\sum_{k\in\Z}  a_k\, e^{\i\, k t}$.
\end{thm}

However, in  \cite{HLN}, for $p<1$ an additional assumption was imposed.
Now we show that it was not necessary, answering the question raised in
\cite[Remark $1.2$]{HLN}.

{\bf Sketch of the proof:}
We have to study the norm of the random vector $Z\in\ell_2(\Z)$ defined by its coordinates
$Z_k= d_k U_k$, $k\in \Z$. It was proved in \cite{HLN} that the corresponding covariance
operator $\KK_Z$ admits a representation
\be\label{KKZ}
    \KK_Z={\bf DAA^*D},
\ee
where $\bf D$ is the convolution operator with the kernel
$\sum_{k \in \Z} d_k\, e^{\i\, k t} $ while $\bf A$ is the multiplication operator related
to the function ${\mathfrak a}(t)$.

We see that the elements of decomposition in (\ref{KKZ})
 are the same as in $(\ref{decomp-cov})$
but the order of use of operators is different.
However, a well-known theorem in operator
theory, see, e.g., \cite[Sec. 2.10, Theorem 5]{BS3}, implies the coincidence of non-zero
eigenvalues for operators $\TT\TT^*$ and $\TT^*\TT$ for any bounded linear operator $\TT$.
Thus, Lemma \ref{Lemma21} implies
that spectral asymptotics of \eqref{s2} type holds for the operator $\KK_Z$
(with the natural replacement $r\to 2p$, $M_{\pm}\to d_{\pm}^2$, $\sqrt{q}\to\mathfrak a$).
Using formula \eqref{sd2} we obtain the claimed small deviation asymptotics.
\hfill $\Box$

\section{Stationary processes with continuous spectra}
\label{s:nonper}
\setcounter{equation}{0}

\subsection{Spectral representations}

Now we consider general aperiodic stationary processes.
Let $X(t),t\in\R$, be a centered second order complex stationary process on $\R$.

The analogue of spectral representation \eqref{specX} is more involved and writes as follows:
\[
   X(t) = \int\limits_{\R}  e^{\i t u} \xi(du),\qquad  t \in \R,
\]
where $\xi(du)$ is an uncorrelated white noise with a control measure $\mu$
called spectral measure of $X$.

The only information about white noise integrals, that we need here is that
the random variable
$\int\limits_{\R}  g(u) \xi(du)$ is well defined and centered
iff $g\in L_2(\R,\mu)$, while for the covariances we have the expression
\[
  \cov\left(\int\limits_{\R} g_1(u)\, \xi(du), \int\limits_{\R}  g_2(u)\, \xi(du) \right)
  = \int\limits_{\R}  g_1(u) \, \overline{g_2(u)}\, \mu(du).
\]
In particular,
\[
   K_X(s-t) := \cov(X(s),X(t))= \int\limits_{\R}  e^{\i u (s-t)}\, \mu(du),
   \qquad s,t\in\R.
\]
We are interested in the small ball behavior of the weighted $L_2$-norm
\[
  \int\limits_{\R} q(t) |X(t)|^2 dt = ||\sqrt{q}\, X||^2_2,
\]
where $q\in L_1(\R)$ is a non-negative weight.

Assume that the spectral measure $\mu$ has a density $m\in L_1(\R)$.
Then it is easy to see that
\[
  (\KK_{\sqrt{q}X}u)(s)=
  \sqrt{q(s)}\int\limits_{\R}K_X(s-t)\sqrt{q(t)}u(t)\,dt
  = 2\pi\sqrt{q(s)}\, \FF^{-1} \big(m(\xi)\FF(\sqrt{q}\,u)(\xi)\big)
\]
(we recall that $\FF$ stands for the Fourier transform).

\subsection{Spectral asymptotics}

From now on we assume that the spectral density $m$ has a power-like decay
analogous to \eqref{muk},
\be \label{muu}
      \lim\limits_{u\to \pm\infty} |u|^r m(u)=M_{\pm},
\ee
with some $r>1$ and $M_{\pm}\ge0$, $M_++M_->0$.

\begin{lem}
Let the spectral density of $X$ satisfy $(\ref{muu})$. Assume that  $q\in L_1(\R)$, and
\be \label{|q|}
 |q|_r:=\sum\limits_{j\in\Z} \|q\|_{L_1(j,j+1)}^{\frac 1r}<\infty.
\ee
Then
\be \label{s2R}
   \lambda_n(\KK_{\sqrt{q}X})=s_n(\KK_{\sqrt{q}X})
   \sim 2\pi\left(\frac{M_-^{\frac 1r}+ M_+^{\frac 1r}}{2\pi}
   \int\limits_{\R} q(t)^{\frac 1r} \,dt\right)^{r}
   \,  n^{-r}, \qquad \textrm{as } n\to\infty.
\ee
\end{lem}

\begin{proof}
  We cannot apply Proposition \ref{BS1} directly since it requires boundedness of the weights supports.
  Therefore, we use subtle estimates of \cite[Sec. 5]{BKS}, see Proposition \ref{BKS}.
  We introduce a decomposition similar to (\ref{decomp-cov}):
  \[
   \KK_{\sqrt{q}X}=\widetilde\TT^*\widetilde\TT, \qquad \widetilde\TT=\MM\FF \QQ,
  \]
where $\MM$ and $\QQ$ stand for the multiplication by $\sqrt{m}\in L_2(\R)$
and $\sqrt{q}\in L_2(\R)$, respectively.

 Following \cite{BKS}, for $f\in L_2(\R)$ we define the numerical sequence
 \be\label{box}
  v(f)=\{v_j(f)\}_{j\in\Z};\qquad  v_j(f):=\|f\|_{L_2(j,j+1)}.
 \ee
Using the notation in Appendix we can write the assumption (\ref{muu}) as follows:
\[
    |\!|\!| v(\sqrt{m})|\!|\!|_{\frac 2r,\infty}<\infty.
\]
Further, the assumption (\ref{|q|}) is equivalent to $v(\sqrt{q})\in \ell_{\frac 2r}$,
and the (quasi)-norm $\|v(\sqrt{q})\|_{\frac 2r}$ coincides with $|q|_r$.

Now we consider the sequence of operators $\widetilde\TT_k=\MM\FF \QQ_k$, $k\in\N$,
where $\QQ_k$ is multiplication by compactly supported weight
  \[
   b_k(t)=\sqrt{q(t)}\cdot{\mathbf 1}_{[-k,k]}(t).
  \]
Obviously, $v(b_k)\to v(\sqrt{q})$ in $\ell_{\frac 2r}$.

Since $\frac 2r<2$, we can apply Proposition \ref{BKS} to the operator
$\widetilde\TT^*-\widetilde\TT_k^*$. This gives
\[
   \sup\limits_n \big(n^{\frac r2} s_n(\widetilde\TT^*-\widetilde\TT_k^*)\big)
   \le const\cdot|\!|\!| v(\sqrt{m})|\!|\!|_{\frac 2r,\infty}\cdot
   \|v(\sqrt{q})-v(b_k)\|_{\frac 2r}\to0 \qquad \textrm{as } k\to \infty.
\]
By (\ref{Nass}) and Proposition \ref{BS3} we infer
\[
   \Delta_{\frac{2}{r}} (\widetilde\TT_k^*)
   \to \Delta_{\frac{2}{r}} (\widetilde\TT^*)
   \qquad \textrm{as } k\to \infty.
\]
Since $\lambda_n(\KK_{\sqrt{q}X})=s_n^2(\widetilde\TT^*)$, this implies
\[
    \Delta_{\frac{1}{r}} (\KK_k)
    \to \Delta_{\frac{1}{r}} (\KK_{\sqrt{q}X})
    \qquad \textrm{as } k\to \infty,
\]
where
\[
    (\KK_k u)(s)= (\widetilde\TT_k^*\widetilde\TT_ku)(s)
    =b_k(s)\int\limits_{\R}K_X(s-t)b_k(t)u(t)\,dt.
\]
The weights $b_k$ satisfy the assumptions of Proposition \ref{BS1}. Using Proposition \ref{BS1}, part {\bf 1} in Proposition \ref{BS2} and the
last statement in Proposition \ref{BS3}, we obtain
\begin{eqnarray*}
   \Delta_{\frac{1}{r}} (\KK_k)
   &=&  \frac 1{2\pi} \int\limits_{\R} \int\limits_{\R\backslash\{0\}}
       \ed{2\pi|b_k(t)|^2\, |M(\sgn(\xi))| \, |\xi|^{-r} \ge 1} \,d\xi dt
   \\
   &=&  \frac 1{2\pi} \int\limits_{-k}^{k}  (2\pi q(t))^{\frac{1}{r}} \,dt
       \left( M(-1)^{\frac{1}{r}}+ M(1)^{\frac{1}{r}} \right)
\end{eqnarray*}
(recall that $M(\pm1)=M_{\pm}$).
We pass to the limit as $k\to\infty$, and the equivalence in \eqref{Nass} yields (\ref{s2R}).
\end{proof}

\subsection{Gaussian small deviations}

\subsubsection{Real processes}

By combining spectral asymptotics \eqref{s2R} with small deviation asymptotics \eqref{sd2}
we immediately obtain the following result.

\begin{thm} \label{t:sd_w_real}
Let $\{X(t), t\in\R\}$ be a  real centered mean-square continuous stationary Gaussian process.
Assume that it has a spectral density satisfying asymptotical condition
\[
    m(u) \sim M |u|^{-r}, \qquad \textrm{as } |u|\to \infty,
\]
with some $r>1, M>0$. Let $q$ be a summable weight satisfying condition $\eqref{|q|}$.

Then we have, as $\eps\to 0$,
\[
   \ln \P\left( \int\limits_{\R} q(t) |X(t)|^2 dt \le \eps^2\right)
   \sim - \left(\frac{M^{\frac 1r}}{r\sin(\pi/r)}
   \int\limits_{\R} q(t)^{\frac 1{r}}  dt\right)^{\frac{r}{r-1}}\,
   \frac{(r-1)(2\pi)^{\frac{1}{r-1}}}{2\, \eps^{\frac{2}{r-1}}} \, .
 \]
\end{thm}

Apart from the weight integration domain, the constant in the limit is exactly the same
as in Theorem \ref{t:sd_w_per_real}.

This result has an intersection with that of S. Gengembre \cite{Gen_IRMA}
who considered the non-weighted $L_p$-norm, $1\le p\le +\infty$, on a bounded interval
and the range $1<r<3$ that enables comparison with fractional Ornstein--Uhlenbeck
processes and thus a reduction to the small deviation results on fractional Brownian
motion, cf. \cite{LifLin3}.
We illustrate this connection in the next sub-section.

\subsubsection{Basic example}

Let $H\in(0,1)$ be the fractionality parameter.
Let $W_H$ be a fractional Brownian motion and
let $U_H(t)=e^{-Ht/2}W_H(e^t), t\in \R$, be a fractional
Ornstein-Uhlenbeck (OU) process.
(There are several other ways to extend the classical OU-process
to the fractional case. We refer to \cite{KS} for alternative definitions and further references.)
In other words, it is a real centered Gaussian stationary process with covariance
\[
   K_H(t)=\frac{1}{2} \left( e^{Ht}+  e^{-Ht} -
   \left| e^{t/2}- e^{-t/2}\right|^{2H} \right).
\]
The asymptotic behavior of the corresponding spectral density $m_H(u)$ is well known,
 see e.g. \cite[Proposition 1]{Gen_IRMA},
\be \label{m_OUH}
   m_H(u) \sim \frac{\Gamma(2H+1)\sin(\pi H)}{2\pi} \ |u|^{-2H-1} =: M_H |u|^{-2H-1},
   \qquad \textrm{as } u\to\infty.
\ee
This is essentially due to the behavior of the covariance at the origin,
\[
    K_H(t)= 1-\frac{|t|^{2H}}{2}+ O(|t|^{1+\min\{1,2H\}}),
     \qquad \textrm{as } t\to 0.
\]
It will be also useful for us to consider integrated versions of fractional Brownian motion
and their stationary versions. Let us denote $V_h := W_h$ for $h\in(0,1)$, and define processes $V_h$
for all non-integer positive $h>1$ inductively, by
\[
    V_{h+1}(t):=\int\limits_0^t V_h(s) ds, \qquad t\ge 0.
\]
It is easy to see that the process $V_h$ is $h$-self-similar. Therefore,
$\{U_h(t)=e^{-ht} V_h(e^t), t\in \R\}$
is a stationary process with the covariance function
\[
    K_h(t)= e^{-ht} \E\left( V_h(1)V_h(e^{t})\right), \qquad t\in \R,
\]
for all positive non-integer values of the parameter $h$.
We can also easily find the inductive formula for the spectral measures of $U_h$. Indeed,
for any $h>1, t\in\R$, we have
\[
  U_h'(t)= - h e^{-ht} V_h(e^t) +  e^{-ht}  e^t V_{h-1}(e^t)
         = -h U_h(t)+U_{h-1}(t).
\]
Rewrite this identity as
\[
   U_{h-1}(t) = U_h'(t) + h U_h(t)
\]
and translate it in the language of spectral measures.
Let $\mu_h$ denote the spectral measure of $U_h$.
Recall that $U_h$ has a spectral representation
\[
   U_h=\int_\R e^{itu} Z_h(du)
\]
where $Z_h$ is a centered random measure with orthogonal values on $\R$
such that $\E|Z_h(\cdot)|^2= \mu_h(\cdot)$.
Since
\[
   U_h'(t) =\int_R e^{itu}\  iu\, Z_h(du),
\]
we get
\[
  U_{h-1}(t)= \int_\R e^{itu}  (iu+h) Z_h(du).
\]
By the uniqueness of the spectral representation,
it follows that $(iu+h) Z_h(du)=Z_{h-1}(du)$
and we finally obtain
\[
   \mu_h(du)= \frac{\mu_{h-1}(du)}{|iu+h|^2}
   = \frac{\mu_{h-1}(du)}{u^2+h^2}\,.
\]
It follows from \eqref{m_OUH}
that $\mu_h$ has a spectral density $m_h$ satisfying
\[
   m_h(u) \sim M_H\,|u|^{-2h-1},
   \qquad \textrm{as } u\to\infty.
\]
(Here and elsewhere $H:=\{h\}$ is the fractional part of $h$).

Assuming condition $\eqref{|q|}$ on the weight to hold and applying Theorem \ref{t:sd_w_real}
with $r=2h+1$, $M=M_H$ we obtain
as $\eps\to 0$,
\begin{multline} 
   \ln \P\left( \int\limits_{\R} q(t) |U_h(t)|^2 dt \le \eps^2\right)
\\ \label{sd_Uh}
  \sim - \left(\frac{1}{(2h+1)\sin(\frac{\pi}{2h+1}))}
   \int\limits_{\R} q(t)^{\frac 1{2h+1}}  dt\right)^{\frac{2h+1}{2h}}\,
   \frac{ h\left( \Gamma(2H+1)\sin(\pi H)  \right)^{\frac 1{2h}}}{\eps^{\frac{1}{h}}}.
\end{multline}
In view of the identity
\[
  \int\limits_{\R} q(t)\, |U_h(t)|^2 dt = \int\limits_{0}^\infty \rho(t)\, |V_h(t)|^2 dt
\]
with the weight
\[
 \rho(t):= \frac{q(\ln t)}{t^{2h+1}}\, , \qquad t>0,
\]
formula \eqref{sd_Uh} immediately yields an equivalent result for the weighted $L_2$-norm of
$V_h$.
The small ball asymptotics for the weighted $L_2$-norm of $V_h$ and $U_h$ was obtained in
\cite[Theorems 3.1, 3.3 and 4.2]{NazNik} but only for the weights with bounded support.

One should also mention \cite{Gen_mem, LifLin3} where small deviations of more general
weighted $L_p$-norms, $1\le p\le+\infty$ were studied for fractional Brownian motions
and for Riemann--Liouville processes.

\subsubsection{Proper complex processes}

In all the previous examples the spectral density $m$ satisfied
\eqref{muu} with $M_+ = M_-$. For complex-valued processes this condition may be violated.
By repeating the proof of Theorem \ref{t:sd_w_per_compl} and using asymptotics
\eqref{s2R} we obtain the following analogue of Theorem \ref{t:sd_w_per_compl}
for complex-valued processes with continuous spectra.

\begin{thm} \label{t:sd_w_compl}
Let $\{X(t), t\in\R\}$ be a complex centered mean-square continuous stationary
proper Gaussian process. Assume that it has a spectral density satisfying
the asymptotic condition \eqref{muu} with some $r>1$. Let $q$ be a summable weight
on $\R$ satisfying \eqref{|q|}.

Then we have, as $\eps\to 0$,
\[
   \ln \P\left( \int\limits_0^{2\pi} q(t) |X(t)|^2 dt \le \eps^2\right)
   \sim - \left(\frac{M_-^{\frac 1r}+ M_+^{\frac 1r}}{2r\sin(\pi/r)}
   \ \int\limits_{\R} q(t)^{\frac 1{r}}  dt\right)^{\frac{r}{r-1}}\,
   \frac{(r-1)(2\pi)^{\frac{1}{r-1}}}{\eps^{\frac{2}{r-1}}} \, .
\]
\end{thm}

Apart from the weight integration domain, the constant in the limit is exactly the same
as in Theorem \ref{t:sd_w_per_compl}.

\appendix

\section{\!\!\!\!\!\!ppendix}

Here we collect some statements from \cite{BKS}--\cite{BS3}.
Recall that $\FF: L_2(\R)\to L_2(\R)$ is standard Fourier transform.
For a compact operator $T$ in $L_2(\R)$ we denote by $s_n(T)$ its singular values and by
\[
  N(\lambda;T) :=  \#\{n: s_n(T)\ge \lambda\}
\]
the distribution function of $s_n(T)$. Define
\[
\Delta_\theta(T) = \lim_{\lambda\to 0_+} \lambda^\theta N(\lambda; T).
\]

Denote by $\ell_{\delta}$ and $\ell_{\delta,\infty}$ the spaces of sequences $(x_j)$ (with $j\in\N$ or $j\in\Z$) such that, respectively,
\begin{equation*}
\|(x_j)\|_{\delta} :=\Big(\sum\limits_j|x_j|^{\delta}\Big)^{\frac 1{\delta}}<\infty;\qquad
|\!|\!|(x_j)|\!|\!|_{\delta,\infty} :=\sup\limits_j\big(|j|^{\frac 1{\delta}} |x_j|\big)<\infty.
\end{equation*}

\begin{prop}\label{BS1}(a particular case of \cite[Theorem 1 (b) and Theorem 2]{BS1}). Let
 \begin{equation*}
  (\A u)(s)= b(s)\FF^{-1} \big(a(\xi)\FF(cu)(\xi)\big),
 \label{bs1}\tag{A.1}
 \end{equation*}
where functions $b,c\in L_2(\R)$ have compact supports while $a$ has the form
\[
  a(\xi)=\zeta(\xi)M(\sgn(\xi))|\xi|^{-r},
 \]
here $r>1$, $M:\{-1,1\}\to [0,+\infty)$ and $\zeta$ is a smooth cut-off
 function vanishing in a neighborhood of the origin. Then
\begin{equation*}
   \Delta_{\frac{1}{r}}(\A)=
\frac 1{2\pi} \int\limits_\R \int\limits_{\R\backslash\{0\}}
       \ed{|b(t)|\,|c(t)|\, |M(\sgn(\xi))| \, |\xi|^{-r} \ge 1} \,d\xi dt.
\end{equation*}

\end{prop}

\begin{prop}\label{BS2}(a particular case of \cite[Corollary 4 and Lemma 3]{BS2}).
Let operator $\A$ have the form (\ref{bs1}).

{\bf 1}. Suppose that weight functions $b$ and $c$ satisfy the assumptions of
Proposition \ref{BS1} while $a(\xi)=\zeta(\xi)\psi(\xi)$, where $\zeta$ is
a smooth cut-off function vanishing in a neighborhood of the origin and
$\psi\in L_\infty(\R)$, $\psi(\xi)=o(|\xi|^{-r})$ as $|\xi|\to\infty$, $r>1$.
Then $\Delta_{\frac{1}{r}}(\A)=0$.

{\bf 2}. Suppose that functions $a$, $b$ and $c$ satisfy the assumptions of
Proposition \ref{BS1}. Let ${\rm supp}(b)\subset I_1$, ${\rm supp}(c)\subset I_2$,
where $I_1$ and $I_2$ are closed bounded segments with non-overlapping interiors.
Then $\Delta_{\frac{1}{r}}(\A)=0$.

\end{prop}

\begin{prop}\label{BKS}(a particular case of \cite[Subsection 5.7]{BKS}). Let
\begin{equation*}
  (\widetilde\A u)(t)= b(t)\FF^{-1} \big(a(\xi)u(\xi)\big),
\end{equation*}
where $a,b\in L_2(\R)$. Define sequences $v(a)$ and $v(b)$ according to (\ref{box}).

Let $v(a)\in \ell_{\delta,\infty}$, $v(b)\in \ell_{\delta}$ for some $\delta\in (0,2)$. Then $\big(s_j(\widetilde\A)\big)\in\ell_{\delta,\infty}$, and
\begin{equation*}
|\!|\!|\big(s_j(\widetilde\A)\big)|\!|\!|_{\delta,\infty}\le const\cdot |\!|\!|v(a)|\!|\!|_{\delta,\infty}\cdot \|v(b)\|_{\delta},
\end{equation*}
where constant depends on $\delta$.
\end{prop}

\begin{prop}\label{BS3}(Corollary 5 in \cite[Sec. 11.6]{BS3}). If $\Delta_\theta(T_1)$, $\Delta_\theta(T_2)$ are finite then
\[
 \big|(\Delta_\theta(T_1))^{\frac 1{\theta+1}}- (\Delta_\theta(T_2))^{\frac 1{\theta+1}}\big|\le  (\Delta_\theta(T_1-T_2))^{\frac 1{\theta+1}}.
\]
In particular, if $\Delta_\theta(T_1-T_2)=0$ then
\[
 \Delta_\theta(T_1)=\Delta_\theta(T_2).
\]

\end{prop}

\subsection*{Acknowledgments}

We are grateful to V. A. Sloushch who provided us with reference \cite{BKS}.
We are also grateful to the anonymous referee and to the Editor for their careful reading and
for the help with our work on the manuscript.

The work was supported by  SPbSU-DFG grant 6.65.37.2017 and by RFBR grant 16-01-00258.



\begin{thebibliography}{1}

\bibitem{AsG}
R.B. Ash, M.F. Gardner, Topics in Stochastic Processes.
\emph{Academic Press}, New York, 1975.

\bibitem{BKS}
M.Sh. Birman, G.E. Karadzhov, M.Z. Solomyak, {\it Boundedness conditions and
spectrum estimates for the operators $b(X)a(D)$ and their analogs}. In:
Estimates and asymptotics for discrete spectra of integral and differential
equations (Leningrad, 1989-90), Adv. Soviet Math., {\bf 7}, AMS, Providence, R.I., 1991, 85--106.




\bibitem{BS1}
M.\v{S}. Birman, M.Z. Solomjak, {\it Asymptotics of the spectrum of
pseudodifferential operators with anisotropic-homogeneous symbols},
Vestnik LGU (1977), no 13, 13--21 (Russian); English transl.: Vestnik Leningrad
Univ. Math. {\bf 10} (1982), 237--247.

\bibitem{BS2}
M.\v{S}. Birman, M.Z. Solomjak, {\it Asymptotics of the spectrum of
pseudodifferential operators with anisotropic-homogeneous
symbols. II}, Vestnik LGU (1979), no 3, 5--10 (Russian).
English transl.: Vestnik Leningrad Univ. Math. {\bf 12} (1980), 155--161.

\bibitem{BS3}
M.S. Birman, M.Z. Solomjak, {\it Spectral Theory of Self-Adjoint Operators
in Hilbert Space}, 2nd edition, revised and extended. Lan', St.Petersburg,
2010 (Russian). English transl. of the 1st ed.: Mathematics and Its
Applications. Soviet Series. {\bf 5}, Kluwer, Dordrecht etc. 1987.

\bibitem{Bro}
J.C. Bronski, {\it Small ball constants and tight eigenvalue asymptotics for
  fractional Brownian motions}, J. Theoret. Probab. \textbf{16} (2003),
  no.~1, 87--100.


\bibitem{DLL} T. Dunker,  M.A. Lifshits,  W. Linde, {\it Small deviations of sums
of independent variables.}
In: {Proc. Conf. High Dimensional Probability}; Ser. Progress in Probability,
{\bf 43}, Birkh\"auser, 1998, 59--74.

\bibitem{GH1}
   Gao, F., Hannig, J., Lee, T.-Y., and Torcaso, F.: Laplace
   transforms via Hadamard factorization with applications to
   small ball probabilities. \emph{Electronic J. Probab.} \textbf{8},
   (2003), paper 13.

\bibitem{GH2}
    Gao, F., Hannig, J., Lee, T.-Y., and Torcaso, F.: Exact
    $L_2$-small balls of Gaussian processes, \emph{J. Theoret. Probab.}
    \textbf{17}, (2004), no.~2, 503--520.

\bibitem{Gen_IRMA}
    S. Gengembre, {\it Probabilit\'es de petites d\'eviations pour les processus
    stationnaires gaussiens}. Publ. IRMA Lille {\bf 60} (2003), no. X, 1--24.

\bibitem{Gen_mem}
    S. Gengembre, {\it Petites d\'eviations pour les processus
    fractionnaires}. Memoire de D.E.A. Universit\'e Lille I, 2002, 19 pp.

\bibitem{HLN}
   S.Y. Hong, M.Lifshits, A.Nazarov, {\it Small deviations in $L_2$-norm for Gaussian
   dependent sequences}, Electronic Comm. Probab. \textbf{21} (2016), no.~41, 1--9.

\bibitem{KS}
   T. Kaarakka, P. Salminen, {\it On fractional Ornstein--Uhlenbeck processes},
   Comm. Stoch. Anal. \textbf{5} (2011), no.~1, 121--133.




\bibitem{LiSha}
  Li, W.V. and Shao, Q.-M.: Gaussian processes: inequalities,
  small ball probabilities and applications, In: Stochastic
  Processes: Theory and Methods, Handbook of Statistics (C.R. Rao
  and D.~Shanbhag, eds.), \textbf{19}, North-Holland/Elsevier,
  Amsterdam, 2001, pp.~533--597.

\bibitem{Lif}
  Lifshits, M.A.: Asymptotic behavior of small ball
  probabilities, In: Probab. Theory and Math. Statist. Proc.
  VII International Vilnius Conference (1998) (B.~Grigelionis, ed.),
  VSP/TEV. Vilnius, 1999, pp.~453--468.

\bibitem{Bib}
  Lifshits, M.A.: Bibliography of small deviation
  probabilities,\quad On the small deviation website
  {\tt http://www.proba.jussieu.fr/pageperso/smalldev/biblio.pdf}

\bibitem{LifLin3}
   M.A. Lifshits, W.~Linde, {\it Small deviations of weighted fractional processes
   and average non-linear approximation}, Trans. Amer. Math. Soc. \textbf{357}
   (2005), 2059--2079.

\bibitem{Naz09}
A.I. Nazarov, {\it Log-level comparison principle for small ball
probabilities}. Statist.\,\&\, Probab. Letters {\bf 79} (2009), no.~4, 481--486.

\bibitem{Naz11}
  A.I. Nazarov, {\it Exact $L_2$-small ball asymptotics of Gaussian processes and the
  spectrum of boundary-value problems}, J. Theor. Probab. {\bf22} (2009), no.~3, 640--665.

\bibitem{NazNik}
  A.I. Nazarov, Ya.Yu. Nikitin, {\it Logarithmic $L_2$-small ball asymptotics for some
  fractional Gaussian processes}, Theory Probab. Appl. {\bf 49} (2004), no.~4, 645--658.

\bibitem{NazPus}
  A.I. Nazarov, R.S. Pusev, Comparison theorems for the small ball probabilities of Gaussian
  processes in weighted $L_2$-norms. Algebra \& Analysis {\bf 25} (2013), no.~3, 131--146
 (Russian); English transl.: St. Petersburg Math. J. {\bf 25} (2014), no.~3, 455--466.

\bibitem{NM}
  F. Neeser and J. Massey, Proper complex random processes with applications to information
  theory, IEEE Transactions on Information Theory {\bf 39} (1993), no.~4, 1293--1302.

\bibitem{Ol}
  E. Ollila, On the circularity of a complex random variable, IEEE Signal Processing Letters
  {\bf 15} (2008), 841--844.

\bibitem{Pus}
  R.S. Pusev, {\it Asymptotics of small deviations of the Bogoliubov processes
  with respect to a quadratic norm}, Theoret. and Math. Phys. {\bf 165} (2010),
  no.~1, 1348--1357.

\bibitem{San}
  D.P. Sankovich, {\it Some properties of functional integrals with respect
  to the Bogoliubov measure}, Theoret. and Math. Phys. {\bf 126} (2001),
  no.~1, 121--135.


\bibitem{Zol}
V.M. Zolotarev, {\it Asymptotic behavior of Gaussian measure in
$\ell_2$}, J. Sov. Math. \textbf{35} (1986), 2330--2334.

\bibitem{Zyg}
A.\,Zygmund. \emph{Trigonometrical Series}, Vol.1, Cambridge University
  Press, Cambridge, 1959.

\end{thebibliography}
\end{document}